\documentclass[11pt,a4paper]{article}
\usepackage{amsthm,amsmath, amssymb, amsfonts, graphicx}
\usepackage[pdftex]{hyperref}
\usepackage[all]{xy}
\usepackage{color}
\newtheorem{theorem}{Theorem}
\newtheorem{lemma}{Lemma}
\newtheorem{corollary}{Corollary}

\newtheorem{prop}{Proposition}

\theoremstyle{remark}
\newtheorem{rem}{Remark}
\newtheorem*{example}{Example}
\newtheorem*{ack}{Acknowledgements}

\newcommand*{\R}{\mathbb{R}}

\newcommand*{\N}{\mathbb{N}}
\newcommand*{\C}{\mathbb{C}}
\newcommand*{\field}{\mathbb{F}}
\renewcommand*{\k}{\mathbb{F}}

\newcommand*{\A}{\mathcal{A}}

\newcommand*{\I}{\mathcal{I}}

\newcommand*{\rk}{\mathrm{rk}}

\newcommand*{\pr}{\mathrm{pr}}

\newcommand*{\stab}{\mathrm{Stab}}

\newcommand*{\F}{\mathcal{F}}
\newcommand*{\G}{\mathcal{G}}

\title{Tensor invariants for certain subgroups of the orthogonal group}
\author{Jan Draisma\footnote{Eindhoven University of Technology and
CWI, Amsterdam. Email: \texttt{j.draisma@tue.nl}. }\and Guus Regts\footnote{CWI, Science Park 123 
1098 XG Amsterdam, the Netherlands.  Email: \texttt{regts@cwi.nl}.}}

\begin{document}

\maketitle
\begin{abstract}
Let $V$ be an $n$-dimensional vector space and let $O_n$ be the orthogonal group.
Motivated by a question of B. Szegedy (B. Szegedy, Edge coloring models and reflection positivity, {\sl Journal of the American Mathematical Society}
 Volume 20, Number 4, 2007), about the rank of edge connection matrices of partition functions of vertex models,
we give a combinatorial parameterization of tensors in $V^{\otimes k}$ invariant under certain subgroups of the orthogonal group.
This allows us to give an answer to this question for vertex models with values in an algebraically closed field of characteristic zero.
\\[.5cm]
\textbf{Keywords:} Edge connection matrix, graph invariant, partition function, orthogonal group, tensor invariants, vertex model.
\end{abstract}

\section{Introduction} \label{sec:Introduction}
Let $\field$ be a field and let $V$ be a
$n$-dimensional vector space over $\field$ equipped with a
nondegenerate symmetric bilinear form
$\langle\cdot,\cdot\rangle$ equivalent to the standard form
on $\field^n$.
We use superscript-$*$ to indicate the duals of
$\field$-vector spaces.
Let $R=S V^*$ be the symmetric algebra generated by $V^*$, which we
identify with the ring of polynomial functions on $V$. 
The orthogonal group $O_n$ is the group of invertible linear transformations of $V$ preserving the bilinear form.
It has a natural action on $R$ and hence on $R^*$.
Let $h\in R^*$ and define 
\begin{equation}
 \stab(h):=\{g\in O_n\mid gh=h\}.	\label{eq:stab}
\end{equation}
The orthogonal group acts on $V^{\otimes k}$ for $k\in \N$.
In this paper we give a combinatorial parameterization of the
space of tensors that are invariant under $\stab(h)$ for
certain $h\in R^*$ when $\k$ is an algebraically closed
field of characteristic zero.
This interpretation can be seen as a generalization of the
Brauer algebra, to which it reduces for $h=0$.

This work is motivated by a question of B. Szegedy concerning the rank of edge connection 
matrices of partition functions of vertex models.
To state his question we introduce some terminology.
Let $\G$ be the (countable) set of isomorphism classes of finite,
undirected graphs, allowing multiple edges, loops, and circles. Here
a circle is a graph with one edge and no vertices. An $\field$-valued
\emph{graph invariant} is a map $f:\G \to \field$. We may think of $f$
as mapping graphs to $\field$ in such a way that isomorphic graphs
are mapped to the same element. This allows us to use the term graph
even when we mean 
isomorphism class of graphs. In particular, we will
somewhat inaccurately speak of the vertex set $VG$ and the edge set $EG$
of an element $G\in \G$.

Throughout this paper we set $\N=\{0,1,2\ldots\}$ and for $n\in \N$,
$[n]$ denotes the set $\{1,\ldots,n\}$.

Let $e_1,\ldots,e_n\in V$ be an orthonormal basis for $V$ and let $x_1,\ldots,x_n\in V^*$ be the associated dual basis.
Then $R$ is just $\field[x_1,\ldots,x_n]$, the polynomial ring in $n$ variables.
Following de la Harpe and Jones \cite{HJ} we call any $h\in R^*$ an \emph{($\field$-valued) $n$-color vertex model}.
The vertex model can be considered as a statistical mechanics model, where
vertices serve as particles, edges as interactions between particles, and colors as states or energy
levels.
The \emph{partition function} of $h$ is the graph invariant $f_h:\G\to \field$ defined by
\begin{equation}
f_h(G)= \sum_{\phi: EG\to [n]} \prod_{v\in VG}
h\left(\prod_{e\in \delta(v)}x_{\phi(e)}\right),
\end{equation}
for any $G\in \G$. Here $\delta(v)$ is the multiset of edges incident with $v$.
Several graph invariants are partition function of vertex models. For example, the number of perfect machings, but
 also the number of homomorphisms into a fixed graph \cite{Szegedy} (for complex valued vertex models).

\begin{rem}\label{rem:coordinatefree}
At first sight, this definition of $f_h$ seems to depend on the choice of
orthonormal basis $e_1,\ldots,e_n \in V$. However, the following more
conceptual interpretation  shows that it really only depends on the
bilinear form and on $h \in R^*=(S V^*)^*$, and that the definition can
in principle be extended to symmetric bilinear forms for which  (over
algebraically non-closed fields) no orthonormal basis exists. A graph
$G=(VG,EG)$ gives rise to a polynomial function $\psi:V^{VG} \to \field$
by sending a tuple $(v_i)_{i \in VG}$ to $n$ to the power the number
of circles in $G$ times the product of the expressions $\langle v_i,
v_j \rangle$ over all non-circle edges $\{i,j\} \in EG$.  Then $\psi$ is
an element of the tensor power $R^{\otimes VG}$ (since coordinate rings
of Cartesian products are tensor products of coordinate rings). Applying
the element $h^{\otimes VG}$ to $\psi$ gives a number, which is nothing
but $f_h(G)$. By construction the function $\psi$ is invariant under
the orthogonal group acting on $V^{VG}$, which implies that $f_{gh}=f_h$
for all $g \in O_n$. Many arguments in this paper have coordinate-free
analogues. For the sake of concreteness, however, we mostly work directly
with formulas such as the one above.
\end{rem}

We now introduce the concept of $k$-fragments, for $k \in \N$.
A \emph{$k$-fragment} is a graph which has $k$ distinct degree-one
vertices labeled $1$ up to $k$.  These labeled vertices are called the
\emph{open ends} of the graph.  The edge connected to an open end is
called a \emph{half edge}.  Let $\F_k$ be the set of all $k$-fragments,
so that $\F_0$ equals $\G$, the set of graphs without labels.
Define a gluing operation $*:\F_k\times \F_{k}\to \G$ as follows: in
the disjoint union of $F,H\in \F_k$, connect the half edges incident
with open ends with identical labels to form single edges (with the
labeled vertices erased); the resulting graph is denoted $F*H \in \F_0$;
see Figure~\ref{fig:gluing}.

\begin{figure}
\begin{center}
\includegraphics[width=.8\textwidth]{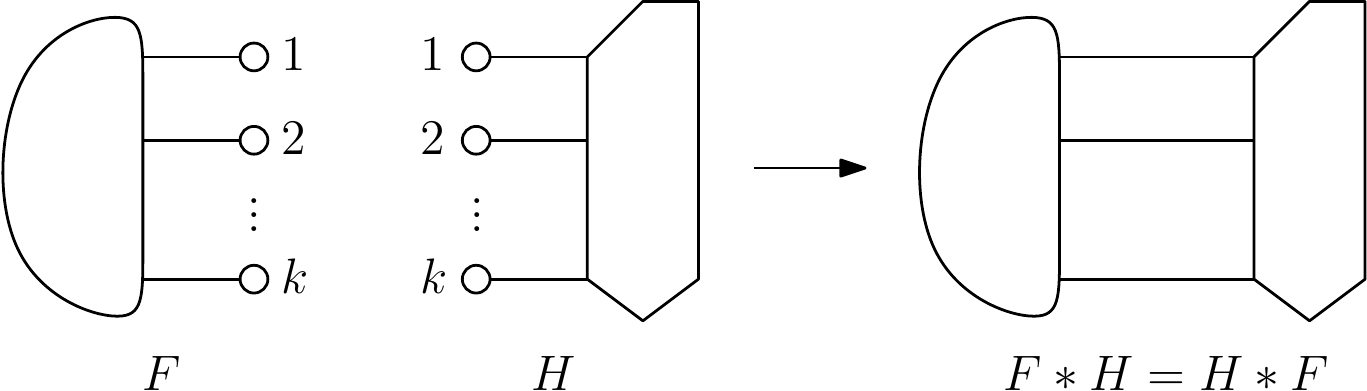}
\caption{Gluing two $k$-fragments into a graph.}
\label{fig:gluing}
\end{center}
\end{figure}

For any  graph invariant $f$ let $M_{f,k}$ be the $\F_k\times \F_k$-matrix defined by 
\begin{equation}
M_{f,k}(F,H)=f(F*H),
\end{equation}
for $F,H\in \F_k$.
This matrix is called the \emph{edge connection matrix} of $f$.
Edge connection matrices were used by Szegedy \cite{Szegedy} to characterize
wich graph invariants $f:\G\to \R$ are partition functions of real valued vertex models and by Schrijver \cite{Sch11}
to characterize wich graph invariants $f:\G\to \C$ are partition functions of complex valued vertex models.

In \cite{Szegedy} Szegedy asked for a characterization of the rank of $M_{f,k}$ 
for $k=1,2\ldots$, if $f=f_h$ for some real valued vertex model $h$. 
This question was answered by the second author in \cite{guus}.
In this paper we give an answer to this question for $\field$-valued vertex models where $\field$ is an 
algebraically closed field of characteristic zero.
Our characterization reads:

\begin{theorem}	\label{thm:main}
Let $\field$ be an algebraically closed field of characteristic zero and let $h\in R^*$.
Then there exists $h'\in R^*$ such that $f_h(G)=f_{h'}(G)$ for all $G\in \G$ and such that
\begin{equation}
\rk(M_{f_h,k})=\dim (V^{\otimes k})^{\stab(h')}.	\label{eq:rank}
\end{equation}
\end{theorem}


The organization of this paper is as follows. 
In Section 2 we compare the real case with the algebraically closed case. 
Section 3 contains some framework and preliminaries. In particular, we state here the the combinatorial parameterization of
$\stab(h')$-invariant tensors (cf. Theorem \ref{thm:main2}) referred to earlier, from which we deduce Theorem \ref{thm:main}. 
In Section 4 we prove a sufficient condition for a subalgebra of the tensor algebra to
be the algebra of invariants of some reductive subgroup of the orthogonal group,
which may be of independent interest. We then use this condition to prove Theorem \ref{thm:main2} in Section 5.
Finally, in Section 6 we prove for a special class of vertex models when we can 
take $h'=h$ in the right-handside of \eqref{eq:rank}.

\section{Real versus complex}
For real valued vertex models the following result holds:
\begin{theorem}[\cite{guus}]	\label{thm:real}
Let $\field =\R$ and let $h\in R^*$. Then
\begin{equation}
\rk(M_{f_h,k})=\dim(V^{\otimes k})^{\stab(h)}.
\end{equation}
\end{theorem}

The proof of this theorem uses the fact that the orthogonal group
of the standard symmetric bilinear form on $\R^n$ is compact;
the analogous statement for non-compact forms is not true. Over
algebraically closed fields of characteristic zero all nondegenerate
forms are equivalent. Consequently, we cannot simply take $h'=h$ in
Theorem \ref{thm:main}, as the following example shows.

\begin{example} \label{ex:nonclosed}
Let $i$ be a square root of $-1$ and set $n:=2, x_1:=x, x_2:=y$.
Consider the vertex model $h:\field[x,y]\to \field$ given by 
\begin{equation}
h(x^ay^b)=\left \{ \begin{array}{l} 1\text{ if } a=1\text{ and }b=0,\\i\text{ if }a=0 \text{ and }b=1\\0, \text{ else. }\end{array}\right.
\end{equation}
Note that for any graph $G$ with at least one vertex we have $f_h(G)=0$. 
Indeed, if $G$ contains an isolated vertex or a vertex of degree at least $2$, then $f_h(G)=0$. Else, $G$ is a perfect matching.  
Since for an edge $e$ we have $f_h(e)=h(x)^2+h(y)^2=0$, also in this
case $f_h(G)$ equals $0$.
So the rank of $M_{f_h,1}$ is equal to zero.
It is not difficult to see that that $\stab(h)=\{I\}$, where $I$ is the identity in $O_2$.
Hence $\rk(M_{f_h,1})\neq \dim V^{\stab(h)}=2$. 
More generally, the following holds: $\rk (M_{f_h,k})=\dim(V^{\otimes k})^{O_2}$.
The vertex model $h'\equiv 0\in \field[x,y]^*$ does the job.

\end{example}

\section{Framework and preliminaries}
In the remainder of this paper we assume that $\field$ is
algebraically closed and has characteristic zero.
Let $\field \F_k$ denote the linear space consisting of (finite) formal $\k$-linear combinations of fragments.
Extend the gluing operation bilinearly to $\k\F_k\times \k\F_k$.
Let
\begin{equation}
\A:=\bigoplus_{k=0}^{\infty}\field\F_k.
\end{equation}
Make $\A$ into a graded associative algebra by defining, for $F\in \F_k$
and $H\in \F_l$, the product $FH$ to be the disjoint union of $F$ and $H$,
where the open end of $H$ labeled $i$ is relabeled to $k+i$.

Fix a vertex model $h\in R^*$. 
Set $\I_k:=\{\gamma \in \k \F_k \mid f_h(\gamma*F)=0\text{ for all $k$-fragments }F\}$
and let $\I:=\bigoplus_{k=0}^\infty \I_k$. 
Observe that 
\begin{equation}
\rk(M_{f_{h},k})=\dim(\k\F_k/\I_k).		\label{eq:rank1}
\end{equation}

Let $T(V):=\bigoplus_{i=0}^{\infty}V^{\otimes k}$ be the tensor algebra of $V$ (with product the tensor product).
For $\phi:[k]\to[n]$ define $e_{\phi}:=e_{\phi(1)}\otimes \cdots\otimes e_{\phi(k)}$. The $e_{\phi}$ form a basis for $V^{\otimes k}$.
We will now exhibit a  natural homomorphism from $\A$ to $T(V)$.

For a $k$-fragment $F$ we denote its edges (including half edges) by $EF$ and its vertices (not including open ends) by $VF$.
Moreover, we will identify the half edges of $F$ with the set $[k]$.
Let $F\in \F_k$ and let $\phi:[k]\to [n]$. 
Define 
\begin{equation}
h_{\phi}(F):=\sum_{\substack{\psi:EF\to [n]\\\psi(i)=\phi(i)\text{ } i=1,\ldots, k}} \prod_{v\in VF}h 
\left(\prod_{e\in \delta(v)}x_{\psi(e)}\right).
\end{equation}
We can now define the map $p_h:\A\to T(V)$ by 
\begin{equation}
F \mapsto \sum_{\phi:[k] \to [n]} h_{\phi}(F) e_{\phi},
\end{equation}
for $F\in \F_k$, extended linearly to $\A$.
Observe that for $F\in \F_0$, $p_h(F)=f_h(F)$,
and note that for $F,H\in \F_k$,
\begin{equation}
f_h(F*H)=\sum_{\phi:[k]\to[n]}h_{\phi}(F)h_{\phi}(H).	\label{eq:prod}
\end{equation}

It is not difficult to see that $p_h$ is a homomorphism of algebras.
By \eqref{eq:prod} it follows that $\ker p_h\subseteq \I$.
This gives rise to the following definition: we call $h\in R^*$ \emph{nondegenerate} if 
$\ker p_h=\I$. Equivalently, $h\in R^*$ is nondegenerate if the
algebra $p_h(\A)$ is nondegenerate with respect to the bilinear form
on $T(V)$ induced by that on $V$.
So for nondegenerate $h$ we have $\A/\I\cong p_h(\A)$ and in particular, by \eqref{eq:rank1},
\begin{equation}
\rk(M_{f_{h},k})=\dim(p_h(\A)\cap V^{\otimes k}).	\label{eq:rank2}
\end{equation}

The following theorem gives a combinatorial parameterization of the tensors invariant under $\stab(h')$.
By \eqref{eq:rank2}, it implies Theorem \ref{thm:main}.

\begin{theorem}	\label{thm:main2}
Let $h\in R^*$.
Then there exists a nondegenerate $h'\in R^*$ with $f_h(G)=f_{h'}(G)$ for all $G\in \G$ such that 
\begin{equation}
p_{h'}(\A)=T(V)^{\stab(h')}.
\end{equation}
\end{theorem}

We will give a proof of Theorem \ref{thm:main2} in Section~5.
The next section deals with some preparations for the proof.

\begin{rem}	\label{rem:invariant}
Following Remark \ref{rem:coordinatefree}, the map $p_h$
can be understood in a coordinate-free manner as follows. As before,
a $k$-fragment $F$ gives rise to a polynomial function $\psi: V^{[k]
\cup VF} \to \field$. Since open ends have degree one, this map is
multilinear in the arguments labeled by $[k]$. Hence $\psi$ sits
naturally in $(V^*)^{\otimes k} \otimes R^{\otimes VF}$. Applying $1
\otimes h^{\otimes VF}$ to it gives an element of $(V^*)^{\otimes k}$,
which is also an element of $V^{\otimes k}$ by the natural map $V^*
\to V$ induced by the bilinear form. This is the element $p_h(F)$.
By definition the function $\psi$ is $O_n$-invariant. This implies that for
all $g\in O_n$ we have $gp_h(F)=p_{gh}(F)$.
\end{rem}

\section{Algebras of invariant tensors}
The proof in \cite{guus} of Theorem \ref{thm:real} depends on a result of Schrijver \cite{Sch} characterizing which subalgebras
 of the real tensor algebra arise as invariant algebras of subgroups of
 the real orthogonal group. Here we prove a variant of
 Schrijver's
 result valid over an algebraically closed field $\field$ of
 characteristic zero.

Below we state and prove a sufficient condition for a subalgebra of
$T(V)$ to be the algebra of $H$-invariants for some
reductive group $H\subseteq O_n$.
Derksen (private communication, 2006) completely characterized
which subalgebras of $T(V)$ are the algebras of
$H$-invariant tensors for some reductive group $H\subseteq
O_n$, but we do not need the full strength of his result
to prove Theorem \ref{thm:main2}.

First we introduce some terminology.  For $1\leq i<j\leq k\in \N$ the
\emph{contraction} $C^k_{i,j}$ is the unique linear map
\begin{align}
&C^k_{i,j}:V^{\otimes k}\to V^{\otimes k-2}	\text{ satisfying}
\\ 
v_1\otimes \ldots \otimes v_k&\mapsto \langle v_i,v_j \rangle v_1\otimes \ldots \otimes v_{i-1}\otimes v_{i+1} \ldots \otimes v_{j-1}\otimes v_{j+1}\otimes \ldots \otimes v_k.	\nonumber
\end{align}
Contractions are clearly $O_n$-equivariant, and so are compositions
of contractions. We will also use $\langle.,.\rangle$ for the induced
symmetric bilinear form on tensor powers $V^{\otimes k}$ of $V$. This
bilinear form is the composition of the tensor product $V^{\otimes k}
\times V^{\otimes k} \to V^{\otimes 2k}$ followed by $C^{2k}_{1,k+1}$,
and $C^{2k-2}_{1,k}$, etc.~up to $C^2_{1,2}$.

A graded subspace $A$ of $T(V)$ is called \emph{contraction closed}
if $C^k_{i,j}(a)\in A$ for all $a\in A\cap V^{\otimes k}$ and $i< j\leq
k\in \N$.  Note that for any subgroup $H\subseteq O_n$, $T(V)^H$ is a
graded and contraction closed subalgebra of $T(V)$.

\begin{theorem}	\label{thm:closed orbit}
Let $A\subseteq T(V)$ be a graded contraction closed subalgebra
containing $\sum_{i=1}^n e_i \otimes e_i$.  If the pointwise stabilizer
$\stab(A):=\bigcap_{a \in A} \stab(a)$ equals $\stab(w)$ for some $w\in
A$ whose $O_n$-orbit is closed in the Zariski topology,
then  $A=T(V)^{\stab(A)}$ and moreover $\stab(A)$ is a reductive group.
\end{theorem}

\begin{proof}
Let $w\in A$ be such that $H:=\stab(w)$ equals $\stab(A)$.
Write $w=w_1+\ldots +w_t$ with $w_j\in W_j:=V^{\otimes n_j}$ the
homogeneous components of $w$, and assume that that $O_n w\subseteq
W:=\bigoplus_{j=1}^t W_j$ is closed. The map $O_n\to W$ given by $g
\mapsto gw$ induces an isomorphism $O_n/H\to O_nw$ of quasi affine
varieties (cf. \cite[Section 12]{Hum}). As $O_nw$ is closed, both varieties are affine and moreover
regular functions on $O_nw$ extend to regular functions (polynomials)
on $W.$ So they are generated by $W_j^*$ for $j=1,\ldots,t$.  This means
that any regular function on $O_n/H$ is a linear combination of functions
of the form
\begin{equation}
gH \mapsto \langle g w_1, u_1 \rangle^{d_1} \cdots \langle g w_t,	
u_t \rangle^{d_t}, \label{eq:reg} 
\end{equation}
where $d_1,\ldots,d_t$ are natural numbers and $u_j \in W_j$
for all $j$. Note that the right-hand side of \eqref{eq:reg} can be obtained
from 
\begin{equation}
 (g w_1^{\otimes d_1}) \otimes \cdots \otimes (g w_t^{\otimes
d_t}) \otimes u_1^{\otimes d_1} \otimes \cdots \otimes
u_t^{\otimes d_t}, \label{eq:ten} 
\end{equation}
by repeatedly contracting the $g w_j$ with the corresponding $u_j$. Since
$A$ is an algebra, the tensor products of the $w_j$
lie in $A$. More succinctly, we find that every regular function on
$O_n/H$ is a linear combination of functions of the form $g \mapsto K
[(gq) \otimes u]=K[q \otimes (g^{-1}u)]$ with $u \in T(V)$ and $q \in A$
in the same graded piece of $T(V)$, and $K$ a composition of contractions.

Clearly, $A$ is contained in $T(V)^H$. To prove the converse, let $a \in
(V^{\otimes k})^H$. Let $z_1,\ldots,z_s$ be a basis of $V^{\otimes k}$. 
Then we can write 
\begin{equation}
ga=\sum_{i=1}^s f_{i}(g)z_i,	\label{eq:ga}
\end{equation}
for all $g\in O_n$, where the $f_{i}$ are regular functions on $O_n$.
Since $gha=ga$ for all $h\in H$ it follows that the $f_{i}$ induce regular
functions on $O_n/H$. By the above, for each $i=1,\ldots, s,$ we can write
\begin{equation}
f_{i}(g)=\sum_j K_{i,j}[q_{i,j} \otimes g^{-1} u_{i,j}],
\end{equation}
for certain $q_{i,j}\in A$ and $u_{i,j}\in T(V)$ and sequences $K_{i,j}$
of contractions.  Multiplying both sides of $\eqref{eq:ga}$ by
$g^{-1}$ we get 
\begin{equation}
a=\sum_{i,j} K_{i,j} [q_{i,j} \otimes (g^{-1} u_{i,j}) \otimes
(g^{-1} z_i)],	\label{eq:a}
\end{equation}
where we have abused the notation $K_{i,j}$ to stand for the same series
of contractions as before, but leaving the last $k$ tensor factors $V$
(containing $g^{-1} z_i$) intact. Let $\rho$ be the Reynolds operator of $O_n$.
Then we have
\begin{equation}
a=\sum_{i,j} K_{i,j} [q_{i,j} \otimes \rho(u_{i,j} \otimes z_i)].		\label{eq:reynolds}
\end{equation}
In the case where $\field=\C$ this follows immediately by integrating
\eqref{eq:a} over $g$ in the compact real orthogonal group (with respect
to the Haar measure). In the general case this follows from standard
properties of the Reynolds operator, which we omit.


To complete the proof note that $q_{i,j}\in A$ and $\rho(u_{i,j}\otimes
z_i)\in T(V)^{O_n}$. Now by the First Fundamental Theorem for the
orthogonal group (see for example \cite[Section 5.3.2]{Wal}), $T(V)^{O_n}$
is the smallest contraction-closed graded subalgebra of $T(V)$ containing
$\sum_i e_i \otimes e_i$, and hence is contained in $A$.  As $A$ is a
graded and contraction closed subalgebra of $T(V)$ it follows that $a\in
A$. Finally, since $O_n/H$ is affine, 
Matsushima's Criterion (see \cite{Arz} for an elementary proof) implies that $\stab(A)=H$
is reductive.
\end{proof}

\section{A proof of Theorem \ref{thm:main2}}
In this section we give a proof of Theorem \ref{thm:main2}.
First we show some results from \cite{guus} allowing us to apply Theorem \ref{thm:closed orbit}.
For completeness we will include the proofs.

We define a contraction operation for fragments.
For $1\leq i< j\leq k\in \N$, the \emph{contraction} $\Gamma^k_{i,j}:\F_k\to\F_{k-2}$ is defined as follows:
for $F\in \F_k$, $\Gamma^k_{i,j}(F)$ is the $(k-2)$-fragment obtained
from $F$ by connecting the half edges incident with the open ends labeled 
$i$ and $j$ into one single edge (without labeled vertex), and then
relabeling the remaining open ends $1,\ldots,k-2$ such that the order is preserved.

The following lemma shows that that $p_h$ preserves contractions.

\begin{lemma}	\label{lem:map p}
For $1\leq i<j\leq k \in \N$ and $F\in \F_k$,  
\begin{equation}
p_h(\Gamma^k_{i,j}(F))=C^k_{i,j}(p_h(F)).
\end{equation}
\end{lemma}

\begin{proof}
Let $1\leq i<j\leq k$ and let $F\in \F_k$.
Note that for $\phi:[k]\to[n]$, the contraction of $e_\phi$ is contained in $\{e_{\psi}\mid \psi:[k-2]\to[n]\}$ if $\phi(i)=\phi(j)$ and is zero otherwise. The following equalities now prove the lemma:
\begin{align}
C^k_{i,j}(p_h(F))&=\sum_{\phi:[k]\to[n]}h_{\phi}(F) C^{k}_{i,j}(e_\phi)=\sum_{\substack{\phi:[k]\to[n]\\ \phi(i)=\phi(j)}}h_{\phi}(F) C^{k}_{i,j}(e_\phi)	\nonumber
\\
&= \sum_{\psi:[k-2] \to [n]}h_{\psi}(\Gamma^k_{i,j}(F))e_\psi=p_h(\Gamma^k_{i,j}(F)).
\end{align}
\end{proof}

The next proposition shows that $\stab(h)$ is equal to the pointwise stabilizer of $p_h(\A)$.

\begin{prop}	\label{prop:stab}
Let $h\in R^*$. Then $\stab(h)=\stab(p_h(\A))$.
\end{prop}

Before we give a proof, we introduce some terminology.
The \emph{basic $k$-fragment $F_k$} is the $k$-fragment that contains one vertex and $k$ open ends connected to this vertex, labeled $1$ up to $k$.
For a map $\phi:[k]\to[n]$, we define the monomial $x^\phi\in R$ by $x^{\phi}:=\prod_{i=1}^k x_{\phi(i)}$. 
It is not difficult to see that
\begin{equation}
h_{\phi}(F_k)=h(x^{\phi}).	\label{eq:hpol}
\end{equation}

\begin{proof}[Proof of Proposition \ref{prop:stab}]
By Remark \ref{rem:invariant} we have for all $F\in \F_k$ and $g\in O_n$ that
\begin{equation}
gp_h(F)=p_{gh}(F).	\label{eq:fragmentinv}
\end{equation}
This immediately implies that $\stab(h)\subseteq \stab(p_h(\A))$. 
To see the opposite inclusion, consider \eqref{eq:fragmentinv} for basic fragments.
Using \eqref{eq:hpol} we find that $g\in \stab(p_h(\A))$ implies that $gh=h$.
\end{proof}

Now we can give a proof of Theorem \ref{thm:main2}.

\begin{proof}[Proof of Theorem \ref{thm:main2}]
The first step of the proof uses some ideas of \cite{DGLRS}.
Consider
\begin{equation}
S:=\field[y_{\alpha}\mid \alpha\in \N^n]
\end{equation}
a polynomial ring in the infinitely many variables $y_{\alpha}$. These
variables are in bijective correspondence with the monomials of $R$
via $y_\alpha \leftrightarrow x_1^{\alpha_1} \cdots x_n^{\alpha_n}$.
Let $\N^n_d=\{\alpha\in \N^n\mid |\alpha|\leq d\}$ and let $S_d\subset S$ be the ring of polynomials in the variables $y_{\alpha}$ with 
$\alpha\in \N^n_d$. 
Furthermore, let $\G_d$ be the set of all graphs of maximum degree $d$.
Let $\pi:\field \G\to S$ be the linear map defined by 
\begin{equation}
G \mapsto \sum_{\phi:EG\to [n]} \prod_{v\in VG} y_{\phi(\delta(v))},
\end{equation}
for any $G\in \G$, where we consider the multiset $\phi(\delta(v))$ as an element of $\N^n$.
Note that $\pi(G)(y)=f_y(G)$ for all $G\in \G$.

The orthogonal group acts on $S$ via the bijection between the variables of $S$ and the monomials of $R$.
Then, as was observed by Szegedy \cite{Szegedy} (see also \cite{DGLRS}), for any $d$, 
\begin{equation} \label{eq:graphsgeninvs}
\pi(\field \G_d)=S_d^{O_n}.
\end{equation}
Let 
\begin{align}
 Y_d:=\{y\in \k^{\N^n_d}\mid \pi(G)(y)=f_h(G) \text{ for all } G\in \G_d\}	\nonumber
\\
Y:=\{y\in \k^{\N^n}\mid \pi(G)(y)=f_h(G) \text{ for all } G\in \G\}
\end{align}
Then $Y$ is nonempty, as it contains $h$. 
Moreover, by \eqref{eq:graphsgeninvs} the variety $Y_d$ is a fiber of
the quotient map $\k^{\N^n_d} \to \k^{\N^n_d}//O_n$. In particular, $Y_d$ 
contains a unique closed orbit $C_d$ (cf. \cite[Section 8.3]{Hum} or \cite[Satz 3, page 101]{kraft}).

Let $\pr_d:\k^{\N^n_{d'}}\to \k^{\N^n_d}$ be the
projection sending $y$ to $y_d:=y|_{\k^{\N^n_d}}$, for any $d'\geq d$.  Note
that $\pr_d(Y_{d'})\subseteq Y_d$ for $d'\geq d$.  Since $\pr_d(C_{d'})$
is an $O_n$-orbit, its closure contains $C_d$. Hence we have 
\begin{equation}
\dim C_d\leq \dim \pr_d(C_{d'})\leq \dim C_{d'},	\label{eq:closed orbits}
\end{equation}
where the left-most inequality is strict unless $\pr_d(C_{d'})$ equals
the closed orbit $C_d$. As $\dim C_d$ is bounded from above by $\dim O_n$
for all $d$, we may choose $d_0$ where the former dimension reaches its
maximal value.  Then for $d' \geq d \geq d_0$ both inequalities in \eqref{eq:closed orbits} are
equalities, and we have $\pr_d(C_{d'})=C_d$. We therefore find a sequence
$h'_{d_0} \in C_{d_0},h'_{d_0+1} \in C_{d_0+1},\ldots$ in which every
next element is projected onto the previous one. That the element $h'\in
\k^{\N^n}$ determined by $\pr_d(h') = h'_{d} \in C_d$ for all $d \geq d_0$
has the property that $f_{h'}=f_h$ follows from the fact that each $h'_d$
lies in $Y_d$. 

We will now show that this $h'$ is as required.  First note that
$\stab(h')=\cap_{e\geq 0} \stab(h'_e)$.  Since the ring of regular
functions of $O_n$ is Noetherian it follows that there exists $e$ such
that $\stab(h')=\stab(h'_e)$. We may assume that $e\geq d_0$.
Let $F=\sum_{0\leq k\leq e}F_k$, the sum in $\A$ of the first $e+1$
basic fragments.  Then
\begin{equation}
\stab(p_{h'}(F))=\stab(h').
\end{equation}
Write $w=p_{h'}(F)$, and note that $w$ is the image of $\pr_e(h')$ under the natural embedding of 
$\k^{\N^n_e}$ into $\bigoplus_{k=0}^e V^{\otimes k}$. 
In particular we can view $C_e$ and $Y_e$ as subvarieties of $\bigoplus_{k=0}^e V^{\otimes k}$. 

It is clear that $p_{h'}(\A)$ is a graded algebra. It is contraction
closed by Lemma \ref{lem:map p}.  By Proposition \ref{prop:stab} we
have that $\stab(p_{h'}(\A))=\stab(w)$.  Moreover, the orbit of $w$ is
Zariski closed, as the orbit of $\pr_e(h')$ is the unique closed orbit
$C_e$ in $Y_e$. Also, we have $\sum_i e_i \otimes e_i\in p_{h'}(\A)$ as 
it is the image of the edge whose both endpoints are open ends.
So we can apply Theorem \ref{thm:closed orbit}, to find that $p_{h'}(\A)=T(V)^{\stab(h')}$.
Moreover, we find that $\stab(h')$ is reductive.
From this we conclude that $h'$ is nondegenerate.
Indeed, suppose that $p_{h'}(\gamma)\neq 0$ for some $\gamma\in \A$.
Then there exists $v\in T(V)$ such that $\langle p_{h'}(\gamma),v\rangle\neq 0$.
Since $\stab(h')$ is reductive we can write $v=v_1+v_2$ with $v_1\in T(V)^{\stab(h')}$ and $v_2$ in a different isotypic component. 
Using Schur's Lemma and the fact that $O_n$ preserves the bilinear form, we find that $\langle p_{h'}(\gamma),v_2\rangle= 0$.
 As $v_1\in p_{h'}(\A)$ it follows that $h'$ is nondegenerate.
\end{proof}

\begin{rem}
The elements $h' \in R^*$ with the property that for sufficiently
large $d$ the projection $h'_d$ lies in the unique closed orbit in the
closure of the orbit of $h_d$ form a single orbit under the orthogonal
group. This follows from the slightly stronger observation that for $d$
sufficiently large and $d' \geq d$ the projection from $C_{d'}$ to $C_d$
is not only surjective and dimension preserving, but even an isomorphism
(since by Noetherianity point stabilizers cannot shrink indefinitely).
\end{rem}

\section{One-parameter groups and spin models}

There is a beautiful criterion for closedness of orbits involving {\em
one-parameter subgroups} of $O_n$, i.e., homomorphisms $\lambda:\k^*
\to O_n$ of algebraic groups. For any such homomorphism there
exists a basis $v_1,\ldots,v_n$ of $V$ such that $\langle v_i,v_j
\rangle=\delta_{n+1-i,j}$ (so that the Gram matrix of the basis has
zeroes everywhere except ones on the longest anti-diagonal; we will call
such bases {\em cano\-nical}) and such that $\lambda(t)v_i=t^{d_i}
v_i$ for some integral {\em weights} $d_1 \geq \cdots \geq d_n$
satisfying $d_i=-d_{n+1-i}$ for all $i$. This follows, for instance,
from \cite[\S 23.4]{Borel} (ignoring the subtle rationality issues
there as $\field$ is algebraically closed) and the fact that all maximal tori are conjugate \cite[\S
11.3]{Borel}. Conversely, given a canonical basis $v_1,\ldots,v_n$
and such a sequence of $d_i$'s, the $\lambda: \k^* \to O_n$ defined by
$\lambda(t)v_i=t^{d_i} v_i$ is a one-parameter subgroup of $O_n$.

The criterion alluded to is that whenever $W$ is a finite-dimensional
$O_n$-module, and $w$ is an element of $W$, there exists a one-parameter
subgroup $\lambda$ such that $\lim_{t \to 0} \lambda(t)w$ exists and lies
in the unique closed orbit in the closure of $O_n w$ (for example, see
\cite[Theorem 6.9]{Popov}). Here the existence of the limit by definition
means that the morphism $\k^* \to W,\ t \mapsto \lambda(t)w$ extends to
$\k$. It then does so in a unique manner, and the value at $0$ is declared
the limit. Put differently, just like $V$ the module $W$ decomposes into
a direct sum of weight spaces, and the condition is that all components
of $w$ in $\lambda$-weight spaces corresponding to negative weights are
zero, and the component of $w$ in the zero weight space is the limit.

\begin{example}
In our example on page \pageref{ex:nonclosed}, $h \in R^* = (\bigoplus_e
S^e V^*)^*$ is zero on all graded pieces $S^e V^*$ except on $S^1
V^*=V^*$. The restriction of $h$ to that space is an element of
$(V^*)^*=V$, namely, equal to $v_1:=e_1 + i e_2$. This is an
isotropic vector relative to the bilinear form, and so is its complex
conjugate $v_2:=e_1-ie_2$. The linear map $V \to V$ scaling $v_1$ with
$t \in \field$ and $v_2$ with $t^{-1}$ is an element of the orthogonal
group. Explicitly, this gives the one-parameter subgroup
\begin{equation}
\lambda: t \mapsto 
\frac{1}{2t}
\begin{bmatrix}
1 + t^2 & i - i t^2\\
-i + it^2 & 1 + t^2
\end{bmatrix} \in O_2
\end{equation}
with the property that $\lim_{t \to 0} \lambda(t) h_e=0$ for
all $0$. 
\end{example}

We will now apply the one-parameter group criterion to an an important class
of vertex models whose partition functions include the partition functions of so-called {\em spin models} (cf. \cite{HJ}) as was shown by B. Szegedy in \cite{Szegedy}. 
Let
$u_1,\ldots,u_m$ be distinct vectors in $V$ and let $a_1,\ldots,a_m$
be nonzero elements of $\k$. Then define
\begin{equation}
h(p):=\sum_{i=1}^m a_i p(u_i), 
\end{equation}
for $p\in R$.
This $h \in R^*$ is a vertex model and we write $h_e$ for the restriction
of $h$ to polynomials of degree at most $e$. We have the following
characterization.

\begin{theorem} \label{thm:Spin}
The orbit $O_n h_e$ is closed for sufficiently large $e$ if and only
if the $O_n$-orbit of $(u_1,\ldots,u_m)$ in $V^m$ is closed, and this
happens if and only if the restriction of the bilinear form to the span
of the $u_i$ is nondegenerate.
\end{theorem}

\begin{proof}
The second equivalence is well known, but we include an argument
as a warm-up for the first equivalence.  First, let $U \subseteq V$
be the span of the $u_i$. If the restriction of the form to $U$ is
degenerate, then we may choose a canonical basis $v_1,\ldots,v_n$ of
$V$ such that that $U$ is spanned by $v_a,\ldots,v_b$ with $b<n+1-a$
(in particular, $v_a$ then lies in the radical of the restriction of
the form to $U$).  Now let $\lambda:\k^* \to O_n$ be the one-parameter
group with $\lambda(t)v_j=t v_j$ for $j \leq a$, $\lambda(t)v_j=v_j$
for $a<j<n+1-a$ and $\lambda(t)v_j=t^{-1}v_j$ for $j \geq n+1-a$.
Then $\lim_{t \to 0} \lambda(t) u$ exists for all $u \in U$ and
lies in the span of $v_{a+1},\ldots,v_b$, a proper subspace of
$U$. Hence $\lim_{t \to 0} \lambda(t)(u_1,\ldots,u_m)$ does not
lie in the orbit of $(u_1,\ldots,u_m)$, and the latter orbit is not
closed. 

For the converse, assume that the restriction of the form to
$U$ is nondegenerate. By the one-parameter group criterion, to prove
closedness of $O_n(u_1,\ldots,u_m)$ it suffices to prove that whenever
$\lambda$ is a one-parameter subgroup of $O_n$ for which the limit
$\lim_{t \to 0}\lambda(t)(u_1,\ldots,u_m)=(u_1',\ldots,u_m')$ exists,
that limit actually lies in the orbit of $(u_1,\ldots,u_m)$. 
Now since the Gram matrix of $(u_1',\ldots,u_m')$
equals that of $(u_1,\ldots,u_m)$, and since its rank equals the
dimension of $U$, we find that the span $U'$ of the $u_i'$ is again
a nondegenerate subspace of $V$ of the same dimension as $U$. Hence
the stabilizers of $(u_1,\ldots,u_m)$ and $(u_1',\ldots,u_m')$ are the
isomorphic groups $O(U^\perp)$ and $O((U')^\perp)$. In particular they have the same
dimension. Since orbits at the ``boundary'' of an orbit have strictly
larger-dimensional point stabilizers, we conclude that
$(u_1',\ldots,u_m')$ is in the $O_n$-orbit of $(u_1,\ldots,u_m)$. 

Now for the first part of the theorem, one direction is easy: if the orbit of
$(u_1,\ldots,u_m)$ is not closed, then there exists a one-parameter
subgroup $\lambda:\k^* \to O_n$ such that $\lambda(t)u_i \to u_i' \in
V,\ i=1,\ldots,m$ for $t \to 0$ and such that $(u_1',\ldots,u_m')$
does not lie in the orbit of $(u_1,\ldots,u_m)$. Then $\lim_{t \to 0}
h_e=:h'_e$ exists and maps any polynomial $p$ of degree at most $e$ to
$\sum_{i=1}^m a_i p(u_i')$. It is not hard to see that for sufficiently
large $e$, the restriction $h'_e$ is {\em not} in the orbit of $h_e$,
roughly because the set of points $u_1,\ldots,u_m$ can be recovered from
$h$ in an $O_n$-equivariant manner (as the set of points defined by the
largest ideal of $R$ contained in the kernel of $h$), and then so can
the $a_i$ by plugging in Lagrange interpolation polynomials at the $u_i$.

For the converse, assume that $u_1,\ldots,u_m$ span a nondegenerate
subspace $U$ of $V$. We will prove that the orbit of $h_e$ is closed for
$e \geq 3m$. Let $\lambda:\k^* \to O_n$ be a one-parameter subgroup such
that $\lim_{t \to 0} \lambda(t)h_e$ exists; we will show that it lies
in the orbit of $h_e$. Let $v_1,\ldots,v_n$ be a canonical basis of $V$
with $\lambda(t)v_j=t^{d_j}v_j$ for weights $d_1 \geq \cdots \geq d_n$.
Let $x_1,\ldots,x_n$ be the basis of $V^*$ dual to $v_1,\ldots,v_n$. For
any monomial $x^\alpha, \alpha \in \N^n$ we have
\begin{equation}
 (\lambda(t) h)(x^\alpha)=h (\lambda(t)^{-1} x^\alpha)=
h(t^{\alpha_1 d_1+\ldots+\alpha_n d_n}x^{\alpha})=
t^{\alpha \cdot d} \sum_{i=1}^m a_i
x^\alpha(u_i).
\end{equation}
By assumption, if $x^\alpha$ is a monomial of degree at most $e$, the
limit for $t \to 0$ of this expression exists. If $\alpha \cdot d=\alpha_1
d_1 + \ldots + \alpha_n d_n$ is negative, this means that $\sum_{i=1}^m
a_i x^\alpha(u_i)$ must be zero. By taking linear combinations, this
implies that for any polynomial $p$ of degree at most $e$ in which
only monomials $x^\alpha$ with $\alpha \cdot d<0$ appear, we have
$h(p)=\sum_{i=1}^m a_i p(u_i)=0$.

In what follows, we exclude the trivial cases where $m=0$
and where $m=1$ and $u_1$ is the zero vector; in both of these
cases, the orbit of $h$ is just a single point. Next let $b \in
\{1,\ldots,n\}$ be the maximal index with $x_b(U) \neq \{0\}$, and
order the $u_i$ such that $x_b(u_1),\ldots,x_b(u_l) \neq 0\ (l>0)$ and
$x_b(u_{l+1}),\ldots,x_b(u_m)=0$. If $d_b$ is nonnegative, then all
$u_i$ lie in the sum of the weight spaces with nonnegative weights, so
that $\lim_{t \to 0}\lambda(t)(u_1,\ldots,u_m)$ exists, and by the second
equivalence we know that it lies in the orbit of $(u_1,\ldots,u_m)$. Then
also $h_e$ and $\lim_{t \to 0} \lambda(t)h_e$ lie in the same orbit. Hence
we may assume that $d_b$ is negative (in particular, $b$ is larger
than $\frac{n}{2}$).

By maximality of $b$, the coordinates $x_{b+1},\ldots,x_n$
vanish identically on $U$, and this means that $U$ lies in the
subspace of $V$ perpendicular to $v_1,\ldots,v_{n-b}$. Since $U$ is
nondegenerate, it does not contain a nonzero linear combination of
$v_1,\ldots,v_{n-b}$. This means, in particular, that the coordinates
$x_{n-b+1},\ldots,x_b$ together separate the points $u_1,\ldots,u_l$.
Then so do the monomials \\$x_b,x_{n-b+1}x_b^2,\ldots,x_{b-1}x_b^2$.
Note that the dot product $\alpha \cdot d$ is negative for each of these
(e.g., for the second, it equals $d_{n-b+1}+2 d_b=d_b<0$ and from
there the dot product decreases weakly to the right).  It follows
that there exists a linear combination $p$ of those (at most) cubic
monomials for which $p(u_1),\ldots,p(u_l)$ are distinct and nonzero.
Then, by the above, the vector $(a_1,\ldots,a_l)^T$ is in the kernel of
the Vandermonde matrix
\begin{equation}
\begin{bmatrix}
p(u_1) & \ldots & p(u_l) \\
p(u_1)^2 & \ldots & p(u_l)^2 \\
\vdots & & \vdots\\
p(u_1)^l & \ldots & p(u_l)^l 
\end{bmatrix},
\end{equation}
since the degree of $p^l$ is $3l \leq e$. Hence $a_1,\ldots,a_l$ are
all zero, contrary to the assumption that all $a_i$ are nonzero. This
proves that the orbit of $h_e$ is closed for $e \geq 3m$.
\end{proof}

An immediate consequence of this result and of the proof of
Theorem~\ref{thm:main} is the following; we omit the proof of the
implication.

\begin{corollary}
If $u_1,\ldots,u_m$ span $V$, $a_1,\ldots,a_m$ are non-zero elements of
$\field$, and $h$ is defined as above, then for all $k$ the
rank of $M_{f_h,k}$ equals $\dim (V^{\otimes k})^H$, where $H$ is the finite group
consisting of all orthogonal transformations of $V$ mapping each $u_i$
to some $u_j$ with $a_j$ equal to $a_i$.
\end{corollary}

\begin{ack}
The second author thanks Lex Schrijver for useful discussions on an
alternative proof for Theorem \ref{thm:main2}. The first author is
supported by a Vidi grant from the Netherlands Organisation for Scientific
Research (NWO).
\end{ack}

\end{document}